\documentclass[a4paper]{article}
\usepackage{fullpage}
\usepackage{amsmath}
\usepackage{amssymb}
\usepackage{amsthm}
\usepackage{graphicx}
\usepackage{bbm}
\usepackage{enumerate}
\usepackage{microtype}
\usepackage{xcolor}
\usepackage[none]{hyphenat}
\usepackage[colorlinks=true,urlcolor=blue,linkcolor=blue,plainpages=false,pdfpagelabels]{hyperref}
\allowdisplaybreaks

\newtheorem{theorem}{Theorem}[section]
\newtheorem{proposition}[theorem]{Proposition}
\newtheorem{lemma}[theorem]{Lemma}

\newcommand{\N}{\mathbb{N}}

\newcommand{\eps}{\varepsilon}

\newcommand{\wt}[1]{\widetilde{#1}}

\title{A quantitative multiparameter mean ergodic theorem}
\author{Andrei Sipo\c s${}^{a,b}$\\[2mm]
\footnotesize ${}^a$Research Center for Logic, Optimization and Security (LOS), Department of Computer Science,\\
\footnotesize Faculty of Mathematics and Computer Science, University of Bucharest,\\
\footnotesize Academiei 14, 010014 Bucharest, Romania\\[1mm]
\footnotesize ${}^b$Simion Stoilow Institute of Mathematics of the Romanian Academy,\\
\footnotesize Calea Grivi\c tei 21, 010702 Bucharest, Romania\\[2mm]
\footnotesize E-mail: andrei.sipos@fmi.unibuc.ro\\
}
\date{}

\begin{document}

\maketitle

\begin{abstract}
We use techniques of proof mining to obtain a computable and uniform rate of metastability (in the sense of Tao) for the mean ergodic theorem for a finite number of commuting linear contractive operators on a uniformly convex Banach space.


\end{abstract}

\section{Introduction}

In the mid-1930s, Riesz proved\footnote{The proof was suggested by an idea of Carleman \cite{Car32} and it first appeared in Hopf's 1937 book on ergodic theory \cite{Hop37}. The argument was published in an individual paper only in 1943 \cite{RS43}.} the following formulation of the classical mean ergodic theorem of von Neumann \cite{vN32}: if $X$ is a Hilbert space, $T: X \to X$ is a linear operator such that for all $x \in X$, $\|Tx\| \leq \|x\|$, then for any $x\in X$, we have that the corresponding sequence of ergodic averages $(x_n)$, where for each $n$,
$$x_n := \frac1{n+1}\sum_{k=0}^n T^kx,$$
is convergent.

It is then natural to ask for a {\it rate of convergence} -- this ties into the area of {\it proof mining}, an applied subfield of mathematical logic. This program in its current form has been developed in the last twenty years primarily by Ulrich Kohlenbach and his collaborators (see \cite{Koh08} for a comprehensive monograph; a recent survey which serves as a short and accessible introduction is \cite{Koh19}), and seeks to apply proof-theoretic tools to results in ordinary mathematics in order to extract information which may not be immediately apparent. Unfortunately, it is known that ergodic convergence can be arbitrarily slow \cite{Kre78}, i.e. there cannot exist a uniform rate of convergence. Kohlenbach's work then suggests that one should look instead at the following (classically but not constructively) equivalent form of the Cauchy property:
$$\forall \eps>0 \,\forall g:\N\to\N \,\exists N\in \N \,\forall i,j\in [N,N+g(N)] \ \left(\| x_i-x_j\| \leq\eps\right),$$
which has been arrived at independently by Terence Tao in his own work on ergodic theory \cite{Tao08A} and popularized in \cite{Tao08} -- as a result of the latter, the property got its name of {\it metastability} (at the suggestion of Jennifer Chayes). Kohlenbach's metatheorems then guarantee the existence of a computable and uniform {\it rate of metastability} -- a bound $\Theta(\eps,g)$ on the $N$ in the sentence above -- extractable from any proof that shows the convergence of a given class of sequences and that may be formalized in one of the logical systems for which such metatheorems have so far been developed, and in the late 2000s, Avigad, Gerhardy and Towsner \cite{AviGerTow10} carried out the actual extraction for Riesz's proof mentioned above.

In 1939, Birkhoff generalized \cite{Bir39} the mean ergodic theorem to uniformly convex Banach spaces; in a paper published in 2009 \cite{KohLeu09}, Kohlenbach and Leu\c stean analyzed that proof in order to obtain a quantitative version.  The main technical obstacle to overcome was the use in the proof of the principle that every sequence of non-negative real numbers $(a_n)$ has an infimum, which they managed to replace by an {\it arithmetical} greatest lower bound principle -- stating that for any $\eps > 0$ there is an $n \in \N$ such that for any $m \in \N$, $a_n \leq a_m + \eps$ -- and gave a quantitative version thereof. Thus, they extracted a rate of metastability which depended, in addition, on the modulus of uniform convexity. In the intervening years, proof mining has continued this line of research, yielding rates of metastability for nonlinear generalizations of ergodic averages \cite{Koh11, Koh12, KohLeu12A, KohLeu12, Saf12, Kor15, LeuNic16, FerLeuPin19, KohSipXX} or bounds on the number of fluctuations \cite{AviRut15} (see \cite{KohSaf14} for a detailed proof-theoretic study of the concept).

Inspired by Birkhoff's proof, Riesz produced a new one in 1941 \cite{Rie41} that delineates more clearly the role played by uniform convexity (i.e. the fact that in such spaces minimizing sequences of convex sets are convergent). One of the advantages of this argument, pointed out in \cite[p. 412]{RS55}, is that it readily generalizes to the case where one deals with more than one contractive operator (a result attributed there to Dunford \cite{Dun39}) -- that is, if $d \geq 1$ and $T_1,\ldots,T_d: X \to X$ are commuting linear operators such that for each $l$ and for each $x \in X$, $\|T_l x\| \leq \|x\|$, then for any $x \in X$, the sequence $(x_n)$, defined, for any $n$, by
$$x_n := \frac1{(n+1)^d} \sum_{k_1=0}^n\ldots\sum_{k_d=0}^n T_1^{k_1}\ldots T_d^{k_d}x$$
is convergent.

{\it The goal of this paper is to extract out of that argument a rate of metastability for the above multiparameter mean ergodic theorem.}

Towards that end, we noticed that the infimum of all the convex combinations of the iterates of $x$ may be effectively replaced by that of just the arithmetic means of pairs of two given ergodic averages. Therefore, we only need to extend the abovementioned principle of Kohlenbach and Leu\c stean to double sequences, which we do in Lemma~\ref{glb-ar-2} by means of the Cantor pairing function. After stating the facts that we shall need about uniformly convex Banach spaces, we express our quantitative metastability result in the form of Theorem~\ref{main}.

\section{Main results}

For all $f:\N \to \N$, we define $\wt{f} : \N \to \N$, for all $n$, by $\wt{f}(n):=n+f(n)$ and $f^M:\N\to\N$, for all $n$, by $f^M(n):=\max_{i \leq n} f(i)$; in addition, for all $n \in \N$, we denote by $f^{(n)}$ the $n$-fold composition of $f$ with itself.

We may now state the quantitative arithmetical greatest lower bound principle of Kohlenbach and Leu\c stean \cite{KohLeu09}.

\begin{lemma}[{cf. \cite[Lemma 3.1]{KohLeu09}}]\label{glb-ar}
Let $(a_n) \subseteq [0,1]$. Then for all $\eps > 0$ and all $g : \N \to \N$ there is an $N \leq \left(g^M\right)^{\left(\left\lceil\frac1\eps\right\rceil\right)}(0)$ such that for all $s \leq g(N)$, $a_N \leq a_s + \eps$.
\end{lemma}

In order to extend the above principle to double sequences, we make use of the bijection $c: \N^2 \to \N$, known as the {\it Cantor pairing function}, defined for all $m$, $n \in \N$ by
$$c(m,n):=\frac{(m+n)(m+n+1)}2 + n.$$
Note that for all $m$ and $n$, we have firstly that $m$, $n\leq c(m,n)$ and then that for all $s$ with $m$, $n\leq s$, $c(m,n) \leq 2s^2+2s$.

\begin{lemma}\label{glb-ar-2}
Let $(a_{m,n}) \subseteq [0,1]$. Define, for any suitable $g$, $\eps$, $s$:
\begin{align*}
f(s)&:=2s^2+2s\\
G(\eps,g)&:=\left(\left(f \circ g\right)^M\right)^{\left(\left\lceil\frac1\eps\right\rceil\right)}(0).
\end{align*}
Then for all $\eps > 0$ and all $g : \N \to \N$ there is an $N \leq G(\eps,g)$ and $p$, $q \leq N$ such that for all $i$, $j \leq g(N)$, $a_{p,q} \leq a_{i,j} + \eps$.
\end{lemma}

\begin{proof}
Put, for all $s$, if $m$ and $n$ are such that $c(m,n)=s$, $b_s:=a_{m,n}$ (using that $c$ is bijective). By Lemma~\ref{glb-ar}, there is an $N \leq G(\eps,g)$ such that for all $s \leq f(g(N))$, $b_N \leq b_s + \eps$. Let $p$ and $q$ be such that $c(p,q)=N$, so $p$, $q \leq N$. Let $i$, $j \leq g(N)$. Since then $c(i,j) \leq f(g(N))$, $b_N \leq b_{c(i,j)} + \eps$, so $a_{p,q} \leq a_{i,j} + \eps$.
\end{proof}

Uniform convexity in Banach spaces was first introduced by Clarkson \cite{Cla36}. We use the following formulation: a Banach space $X$ is {\it uniformly convex} if there is an $\eta : (0,\infty) \to (0,1]$, called a {\it modulus of uniform convexity}, such that for all $\eps>0$ and all $x,y \in X$ with $\|x\|\leq 1$, $\|y\| \leq 1$ and $\|x-y\| \geq \eps$ one has that
$$\left\|\frac{x+y}2\right\| \leq 1-\eta(\eps).$$
One typically defines the fixed modulus of uniform convexity $\delta_X$ by putting, for all $\eps \in (0,2]$,
$$\delta_X(\eps):=  \inf \left\{ 1- \left\| \frac{x+y}2 \right\| \bigm| \|x\|\leq 1, \|y\|\leq 1, \|x-y\| \geq \eps \right\}$$
and says that $X$ is uniformly convex if and only if for all $\eps \in (0,2]$, $\delta_X(\eps)>0$. It is immediate that the two definitions coincide: in this case, $\delta_X$ may serve as a modulus $\eta$ in our sense (it is, in fact, the largest such modulus) by putting for all $\eps>0$, $\eta(\eps):=\delta_X(\min(\eps,2))$.

We shall use the following characteristic property of uniformly convex spaces.

\begin{proposition}[{\cite[Lemma 3.2]{KohLeu09}}]\label{u-prop}
Let $X$ be a Banach space with modulus of uniform convexity $\eta$. Define, for any $\eps > 0$, $u(\eps):= \frac\eps2 \cdot \eta(\eps)$. Then, for any $\eps>0$ and any $x$, $y \in X$ with $\|x\| \leq \|y\| \leq 1$ and $\|x-y\| \geq \eps$ we have that
$$\left\|\frac{x+y}2\right\| \leq \|y\|- u(\eps).$$
In addition, if there is a nondecreasing function $\eta'$ such that for all $\eps$, $\eta(\eps)=\eps\eta'(\eps)$ (e.g. in the case of Hilbert spaces where $\eps \mapsto \frac{\eps^2}8$ is a modulus of uniform convexity), then one can take $u$ to be simply $\eta$.
\end{proposition}

With that in mind, we may now state our main theorem.

\begin{theorem}\label{main}
Let $X$ be a Banach space. Let $G$ be defined as in Lemma~\ref{glb-ar-2}. 
Define, for any suitable $\eps$, $g$, $u$, $d$, $\delta$, $Q$, $\gamma$, $n$:
\begin{align*}
\Phi(d,\delta,Q)&:=\max\left(Q,\left\lceil\frac{2^dQ}\delta\right\rceil\right)\\
h_{\gamma,g,d}(n)&:=\wt{g}\left(\Phi\left(d,\frac\gamma2,n\right)\right)\\
\Psi(d,\gamma,g)&:=\Phi\left(d,\frac\gamma2,G\left(\frac\gamma2,h_{\gamma,g,d}\right)\right)\\
\Theta_{u,d}(\eps,g)&:=\Psi\left(d,\frac{u(\eps)}2,g\right).
\end{align*}
Let $u$ be such that the property described by Proposition~\ref{u-prop} holds. Let $d \geq 1$ and $T_1,\ldots,T_d: X \to X$ be commuting linear operators such that for all $l$, $T_l$ is contractive, i.e. for all $x \in X$, $\|T_l x\| \leq \|x\|$. We denote, for any $x \in X$ and $i \in \N^d$, $T^ix:=T_1^{i_1}\ldots T_d^{i_d} x$. Let $x \in X$ with $\|x\| \leq 1$. Put, for all $n$,
$$x_n := \frac1{(n+1)^d} \sum_{k \in [0,n]^d} T^kx.$$
Let $\eps > 0$ and $g: \N \to \N$. Then there is an $N \leq \Theta_{u,d}(\eps,g)$ such that for all $i$, $j \in [N, N+ g(N)]$, $\|x_i - x_j\| \leq \eps$.
\end{theorem}

\begin{proof}
We first prove the following two claims.\\[1mm]

\noindent {\bf Claim 1.} For all $\delta >0$ and all $Q$ and $(c_j)_{j \in [0,Q]^d} \subseteq [0,1]$ with $\sum_{j \in [0,Q]^d} c_j = 1$, if we set
$$z:=\sum_{j \in [0,Q]^d} c_j T^j x,$$
then for all $n \geq \Phi(d,\delta,Q)$, $\|x_n\| \leq \|z\|+\delta$.\\[1mm]
\noindent {\bf Proof of claim 1:} Put, for all $n$,
$$z_n := \frac1{(n+1)^d} \sum_{i \in [0,n]^d} T^iz.$$
Then it is enough to show that for all $n \geq Q$ we have that $\|x_n-z_n\| \leq \frac{2^dQ}{n+1}$, since then for all $n \geq \Phi(d,\delta,Q)$, given that $n \geq \frac{2^dQ}\delta$, one has $\frac{2^dQ}{n+1} \leq \delta$, so $\|x_n-z_n\| \leq \delta$ and (using here the contractivity of the $T_l$'s)
$$\|x_n\| \leq \|z_n\| + \|x_n - z_n\| \leq \|z\| + \delta.$$

Let, then, $n \geq Q$. We have that (using here that the $T_l$'s are linear and commuting)
$$z_n = \frac1{(n+1)^d} \sum_{i \in [0,n]^d} T^i\sum_{j \in [0,Q]^d} c_j T^j x = \frac1{(n+1)^d} \sum_{i \in [0,n]^d}\sum_{j \in [0,Q]^d}c_j T^{i+j} x.$$
If for any $k \in [0,n+Q]^d$ we denote by $R_k$ the set of all $j \in [0,Q]^d$ with the property that there is an $i \in [0,n]^d$ with $i+j=k$, we have that
$$z_n = \frac1{(n+1)^d} \sum_{k \in [0,n+Q]^d} \left( \sum_{j \in R_k} c_j \right) T^k x.$$
We show that for all $k \in [Q,n]^d$, $R_k=[0,Q]^d$. Let $j \in [0,Q]^d$ and put $i:=k-j$. We have to show that $i \in [0,n]^d$. Let $s \in [1,d]$. Since $Q \leq k_s \leq n$ and $-Q \leq -j_s \leq 0$, we have $0 \leq i_s \leq n$, which is what we needed. Thus, we have
$$z_n = \frac1{(n+1)^d} \sum_{k \in [0,n+Q]^d \setminus [Q,n]^d} \left( \sum_{j \in R_k} c_j \right) T^k x + \frac1{(n+1)^d} \sum_{k \in [Q,n]^d} \left( \sum_{j \in [0,Q]^d} c_j \right) T^k x.$$
Since on the other hand,
$$x_n = \frac1{(n+1)^d} \sum_{k \in [0,n]^d} T^kx = \frac1{(n+1)^d} \sum_{k \in [0,n]^d} \left( \sum_{j \in [0,Q]^d} c_j \right) T^k x,$$
we have, since $[Q,n]^d\subseteq [0,n]^d$, that 
$$x_n-z_n=\frac1{(n+1)^d} \sum_{\substack{k \in [0,n+Q]^d \setminus [Q,n]^d\\ k \in [0,n]^d}} \left( \sum_{\substack{j \in [0,Q]^d\\j \notin R_k}} c_j \right) T^k x - \frac1{(n+1)^d} \sum_{\substack{k \in [0,n+Q]^d \setminus [Q,n]^d\\ k \notin [0,n]^d}} \left( \sum_{j \in R_k} c_j \right) T^k x,$$
so, using that $\|x\| \leq 1$ and (again) the contractivity of the $T_l$'s,
$$\|x_n-z_n\| \leq \frac1{(n+1)^d} \sum_{\substack{k \in [0,n+Q]^d \setminus [Q,n]^d\\ k \in [0,n]^d}} 1 + \frac1{(n+1)^d} \sum_{\substack{k \in [0,n+Q]^d \setminus [Q,n]^d\\ k \notin [0,n]^d}} 1  = \frac{\left|[0,n+Q]^d \setminus [Q,n]^d\right|}{(n+1)^d},$$
so, by putting $m:=n+1$,
$$\|x_n-z_n\| \leq \frac{(n+Q+1)^d - (n-Q+1)^d}{(n+1)^d} = \frac{(m+Q)^d - (m-Q)^d}{m^d}.$$
Since $(m+Q)^d = m^d + \binom{d}{1}m^{d-1}Q + \binom{d}{2}m^{d-2}Q^2 + \binom{d}{3}m^{d-3}Q^3 + \ldots$ and $(m-Q)^d = m^d - \binom{d}{1}m^{d-1}Q + \binom{d}{2}m^{d-2}Q^2 - \binom{d}{3}m^{d-3}Q^3 + \ldots$, we have that
$$(m+Q)^d - (m-Q)^d = 2\left(\binom{d}{1}m^{d-1}Q + \binom{d}{3}m^{d-3}Q^3 + \ldots\right).$$
Since $m \geq Q$, for each $k$, $m^{d-k}Q^k \leq m^{d-1}Q$, so
$$(m+Q)^d - (m-Q)^d \leq 2 m^{d-1}Q\left(\binom{d}{1} + \binom{d}{3} + \ldots\right) = 2^dm^{d-1}Q,$$
from which we get
$$\|x_n-z_n\| \leq \frac{2^dm^{d-1}Q}{m^d} = \frac{2^dQ}m = \frac{2^dQ}{n+1},$$
which is what we wanted.
\hfill $\blacksquare$\\[2mm]

\noindent {\bf Claim 2.} For all $\gamma > 0$ there is an $N \leq \Psi(d,\gamma,g)$ such that for all $i$, $j \in [N, N+ g(N)]$, $\|x_i\| \leq \left\|\frac{x_i+x_j}2\right\| + \gamma$ -- note that by symmetry we also have  $\|x_j\| \leq \left\|\frac{x_i+x_j}2\right\| + \gamma$.\\[1mm]
\noindent {\bf Proof of claim 2:} By Lemma~\ref{glb-ar-2}, there is a $Q \leq G\left(\frac\gamma2,h_{\gamma,g,d}\right)$ and $p$, $q \leq Q$ such that for all $i$, $j \leq h_{\gamma,g,d}(Q)$,
$$\left\|\frac{x_p+x_q}2\right\| \leq \left\|\frac{x_i+x_j}2\right\| + \frac\gamma2.$$
Put $N:=\Phi\left(d,\frac\gamma2,Q\right)$. Then, using that $\Phi$ is nondecreasing in its last argument,
$$N \leq \Phi\left(d,\frac\gamma2,G\left(\frac\gamma2,h_{\gamma,g,d}\right)\right) = \Psi(d,\gamma,g).$$
Let $i$, $j \in [N, N+ g(N)]$. Note that
$$N + g(N) = \wt{g}\left(\Phi\left(d,\frac\gamma2,Q\right)\right) = h_{\gamma,g,d}(Q),$$
so
$$\left\|\frac{x_p+x_q}2\right\| \leq \left\|\frac{x_i+x_j}2\right\| + \frac\gamma2.$$
In addition, since $i \geq \Phi\left(d,\frac\gamma2,Q\right)$, we have by the previous claim that
$$\|x_i\| \leq \left\|\frac{x_p+x_q}2\right\| + \frac\gamma2.$$
Putting together the last two inequalities, we obtain our conclusion.
\hfill $\blacksquare$\\[2mm]

Apply the above claim for $\gamma:=\frac{u(\eps)}2$ to get an $N \leq \Theta_{u,d}(\eps,g)$. Let $i$, $j \in [N, N+ g(N)]$ and assume w.l.o.g. $\|x_j\| \leq \|x_i\|$. Assume that $\|x_i - x_j\| > \eps$. Then, using Proposition~\ref{u-prop} and the conclusion of the above claim, we get
$$\left\|\frac{x_i+x_j}2\right\| \leq \|x_i\| - u(\eps) \leq \left\|\frac{x_i+x_j}2\right\| - \frac{u(\eps)}2,$$
a contradiction.
\end{proof}

Let us extend the above result to an arbitrary $x \in X$. For any $b>0$, set $\Theta_{u,d,b}(\eps,g):=\Theta_{u,d}(\eps/b,g)$. Then if $b>0$ is such that $\|x\| \leq b$, if we put $\wt{x}:=x/b$ we see that $\|\wt{x}\|\leq 1$ and for each $n$, the corresponding $\wt{x}_n$ is equal to $x_n/b$. By applying the above theorem for $\eps/b$, $g$ and $\wt{x}$ we get that there is an $N \leq \Theta_{u,d,b}(\eps,g)$ such that for all $i$, $j \in [N, N+ g(N)]$, $\|\wt{x}_i - \wt{x}_j\| \leq \eps/b$ -- that is, $\|x_i-x_j\|\leq\eps$. Thus, $\Theta_{u,d,b}$ is a rate of metastability applicable for this general situation.

We shall now see a concrete example of the bound in Theorem~\ref{main}. The simplest non-Hilbert examples of uniformly convex Banach spaces are the $L^p$ spaces, and there one can take (as seen in \cite{KohLeu12}), for $p \geq 2$, the function $u$ to be $\eps \mapsto \frac{\eps^p}{p2^p}$. We will take $d$ to be $2$, $\eps <1$ (to simplify the computations) and $g$ to be the function constantly equal to an $L \in \N$ -- leading to what \cite{LeuPin21} calls a {\it rate of $L$-metastability}. Thus, we see that
$$\Theta(\eps)=\Psi\left(2,\frac{\eps^p}{p2^{p+1}},g\right) = \Phi\left(2,\frac{\eps^p}{p2^{p+2}},G\left(\frac{\eps^p}{p2^{p+2}},h_{\frac{\eps^p}{p2^{p+1}},g,2}\right)\right) = \left\lceil\frac{2^{p+4}p G\left(\frac{\eps^p}{p2^{p+2}},h_{\frac{\eps^p}{p2^{p+1}},g,2}\right)}{\eps^p}\right\rceil$$
and, for any $n \in \N$,
$$h_{\frac{\eps^p}{p2^{p+1}},g,2}(n) = \Phi\left(2,\frac{\eps^p}{p2^{p+2}},n\right) + L =  \left\lceil\frac{2^{p+4}np}{\eps^p}\right\rceil+L.$$
Denoting the function $f \circ h_{\frac{\eps^p}{p2^{p+1}},g,2}$ by $j$, we remark that it takes an $n \in \N$ into
$$2\left(\left\lceil\frac{2^{p+4}np}{\eps^p}\right\rceil+L\right)\left(\left\lceil\frac{2^{p+4}np}{\eps^p}\right\rceil+L+1\right)$$
and thus $j^M=j$. Thus, the term above involving $G$,
$$G\left(\frac{\eps^p}{p2^{p+2}},h_{\frac{\eps^p}{p2^{p+1}},g,2}\right),$$
becomes
$$j^{\left(\left\lceil\frac{p2^{p+2}}{\eps^p}\right\rceil\right)}(0),$$
so we may write the rate of $L$-metastability using only the auxiliary function $j$, as 
$$\Theta(\eps)=\left\lceil\frac{2^{p+4} \cdot p \cdot j^{\left(\left\lceil\frac{p2^{p+2}}{\eps^p}\right\rceil\right)}(0)}{\eps^p}\right\rceil.$$

\section{Acknowledgements}

This work has been supported by the German Science Foundation (DFG Project KO 1737/6-1) and by a grant of the Romanian Ministry of Research, Innovation and Digitization, CNCS/CCCDI -- UEFISCDI, project number PN-III-P1-1.1-PD-2019-0396, within PNCDI III.

\end{document}